\documentclass[11pt]{amsart}
\usepackage{a4wide}
\usepackage{amssymb}
\usepackage{mathrsfs}
\usepackage{enumerate}
\theoremstyle{plain}
\newtheorem{theorem}{Theorem}%[section]
\newtheorem*{theorem*}{Theorem}%[section]
\newtheorem{lemma}[theorem]{Lemma}%[section]
\newtheorem{proposition}[theorem]{Proposition}
\theoremstyle{definition}
\newtheorem*{remark*}{Remark}
\newtheorem{remark}[theorem]{Remark}
\newcommand{\NN}{\mathbf{N}}

\begin{document}
\title[Normal generation of locally compact groups]{Normal generation of locally compact groups}
\author[A. Eisenmann and N. Monod]{A. Eisenmann and N. Monod}
\address{EPFL, Switzerland}
\thanks{Supported in part by the European Research Council and the Swiss National Science Foundation.}
\begin{abstract}
It is a well-known open problem since the 1970s whether a finitely generated perfect group can be normally generated by a single element or not. We prove that the topological version of this problem has an affirmative answer as long as we exclude infinite discrete quotients (which is probably a necessary restriction).
\end{abstract}
\maketitle

%===================================================================================================
%===================================================================================================
\section{Introduction}
Let $G$ be a perfect group. Can $G$ be \emph{normally generated} by a single element?

\smallskip

Equivalently: can $G$ be made trivial by adding a single relation? In other words yet: does $G$ contain some element not belonging to any proper normal subgroup?

\bigskip
This is well-known to be true when $G$ is \emph{finite} (see e.g.~\cite[4.2]{Lennox-Wiegold} for a stronger fact). For infinite groups, it is a long-standing open problem attributed to J.~Wiegold; see~\cite[FP14]{Baumslag-Myasnikov-Shpilrain} and~\cite[5.52]{Mazurov-Khukhro}. The group $G$ is assumed to be finitely generated in order to avoid obvious counter-examples (infinite direct sums of perfect groups). It is widely believed to be false: for instance, Lennox and Wiegold conjecture that a free product of $n$ finite groups cannot be normally generated by less than $n/2$ elements~\cite{Lennox-Wiegold}. A completely different potential source of counter-examples is proposed in~\cite{Monod-Ozawa-Thom}.

\bigskip
The question can be asked more generally for topological groups $G$ (which we always suppose Hausdorff). Here $G$ is called \emph{perfect}  if it has no non-trivial continuous homomorphism to an abelian topological group, or equivalently if its commutator subgroup is dense. Likewise, it is deemed \emph{normally generated} by a single element $g\in G$ if the smallest closed normal subgroup containing $g$ is $G$ itself.

\smallskip
If $G$ has an infinite discrete quotient group, then the putative counter-examples for abstract groups will of course be \emph{a fortiori} counter-examples for $G$ itself.  Rather surprisingly, the question has a positive answer when we exclude such quotients:

\begin{theorem}\label{thm:main}
Let $G$ be a compactly generated locally compact group without infinite discrete quotient. If $G$ is perfect, then it is normally generated by a single element.
\end{theorem}

This result uses structure theory, both classical~\cite{Hofmann-Morris, Montgomery-Zippin} and very recent~\cite{Caprace-Monod_compact}.

\begin{remark}\label{rem:ass}
Compact generation is the replacement for the necessary assumption of finite generation. It remains necessary even when there is no infinite discrete quotient, as shown by an example constructed for us by Y.~de Cornulier, see Section~\ref{sec:Yves} below.
\end{remark}

\subsection*{Acknowledgements}
Our motivation for this work came from discussions with Taka Ozawa and Andreas Thom, as always pleasant and inspiring. We are grateful to Yves de Cornulier for allowing us to present his example for Remark~\ref{rem:ass}.

%===================================================================================================
%===================================================================================================
\section{Proof of the Theorem}
We begin with a few simple facts. It is rather straightforward to go from finite groups to profinite groups:

\begin{lemma}\label{lem:pro}
A perfect profinite group is normally generated by a single element.
\end{lemma}

\begin{proof}
Let $G=\varprojlim G/G_{\alpha}$, be a perfect profinite group, where $G_{\alpha}$ runs over the open normal subgroups of $G$. Each $G/G_{\alpha}$ is a perfect finite group and thus it is normally generated by a single element. Denote by $X_\alpha$ the set of elements of $G$ whose image normally generates $G/G_{\alpha}$. Thus $X_\alpha$ is a non-empty and clopen subset of $G$. The family of these sets is directed and thus its intersection is non-empty by compactness of $G$. We claim that any element $g$ in this intersection normally generates $G$. Indeed, suppose that $g$ is contained in a closed normal subgroup $N$ of $G$. In a profinite group, every closed subgroup is an intersection of open subgroups. Since moreover open subgroups have finitely many conjugates, we see that $N$ is an intersection of open normal subgroups. By construction, $g$ does not belong to any $G_\alpha\neq G$. It follows $N=G$, as claimed.
\end{proof}

We shall repeatedly but tacitly use the following fact: if $S$ is a subset of a topological group $G$ and $N\lhd G$ is a closed normal subgroup, then the image of $S$ in $G/N$ is dense (if and) only if $SN$ is dense in $G$. Indeed, the quotient map $G\to G/N$ is open.

\begin{lemma}\label{lem:sol}
Let $G$ be a perfect topological group, and let $N$ be a closed normal soluble subgroup of $G$.
If $G/N$ is normally generated by a single element, then so is $G$.
\end{lemma}

\begin{proof}
Let $g\in G$ be such that its image in the quotient $G/N$ normally generates it. Let $H$ be the minimal closed normal subgroup of $G$ containing $g$. Then $HN$ is dense in $G$. Thus, the canonical homomorphism $N\to G/H$ has dense image, which implies that $G/H$ is soluble. Being also perfect, it must be trivial. Thus $H=G$ as needed.
\end{proof}

We endow the product of any family of topological groups with the product topology. We recall that all but finitely many factors must be compact if the product is locally compact --- for instance, if it appears as a closed subgroup of a locally compact group. In order to use structure theory, we need the following.

\begin{proposition}\label{prop:simple}
Let $G$ be a perfect topological group. Let $H$ be a closed normal subgroup of $G$ which is the topological direct product of topologically simple topological groups. If $G/H$ is generated by a single element, then so is $G$.
\end{proposition}

\begin{proof}
Let $H=\prod_{i\in I}S_i$, where $S_i$ is  a simple topological group for each $i$ in some index set $I$. We can assume that no $S_i$ is abelian by applying Lemma~\ref{lem:sol} to the product of all abelian factors $S_i$. Let $N$ be any closed normal subgroup of $G$. We denote by $[N,H]$ the closed subgroup generated by commutators $[n,h]$ with $n\in N$ and $h\in H$. Hence $[N,H]$ is a closed normal subgroup of $G$ and thus also of $H$. Therefore, there is a subset $J$ of $I$ such that $[N,H]=\prod_{i\in J}S_i$. Let $H_1:=\prod_{i\in I\setminus J}S_i$. Then $H=[N,H]\times H_1$ and the action of $N$ by conjugation induces a permutation on $I$, permuting the simple subgroups $S_i$. As $[N,H]$ is normalized by $N$, this permutation action  on $I$ splits into actions on $J$ and on $I\setminus J$. In particular, $N$ normalizes $H_1$. Thus $[N,H_1]\subset H_1$. But on the other hand $[N,H_1]\subset [N,H]$ and thus $[N,H_1]$ is trivial since $[N,H]\cap H_1$ is trivial. In other words, $N$ centralizes $H_1$.

Let now $g$ be an element of $G$ whose image in $G/H$ is normally generating. We apply the above argument to the smallest closed normal subgroup $N$ of $G$ containing $g$. This yields a dense normal subgroup $NH=N H_1$ in $G$, with $N$ and $H_1$ centralizing each other. Choose an element $h\in H_1$ that is non-trivial in each coordinate $S_i$ when $i\in I\setminus J$. Thus the group $H_1$ is normally generated by $h$ by a standard argument: if $L\lhd H_1$ is the minimal closed normal subgroup containing $h$, then for all $i\in I\setminus J$,  $[L, S_i]$ is a non-trivial closed normal subgroup of $S_i$ and hence equals $S_i$. Therefore $L$ contains $S_i$ and thus indeed $L=H_1$.

We claim that $gh$ normally generates $G$. Let thus $M\lhd G$ be the minimal closed normal subgroup containing $gh$. The image of $M$ in $G/N$ contains $ghN=hN$. Since the image of $H_1$ in $G/N$ is dense, we conclude that the image of $M$ is dense too. That is, $NM$ is dense in $G$. Turning things around, the image of $N$ in $G/M$ is dense and therefore the image of $H_1$ in $G/M$ is central. Since $H_1$ is perfect, this image of $H_1$ is in fact trivial, and thus $h\in M$. It now follows $g\in M$, hence $N<M$ and thus $NM=M$ which finally implies $M=G$ as claimed.
\end{proof}

\bigskip
We now prove Theorem~\ref{thm:main}. Let $G$ be a compactly generated locally compact group without infinite discrete quotients. We begin with some classical structure theory. Because of the solution of Hilbert's fifth problem, there exists a maximal normal compact subgroup $K$ in the identity component $G^\circ$ of $G$ and the quotient $G^\circ/K$ is a connected Lie group; see Lemma~2.2 in~\cite{Caprace-Monod_compact}. Since $K$ is necessarily unique, it is also characteristic in $G$. Accordingly, its identity component $K^\circ$, the commutator subgroup $[K^\circ, K^\circ]$ and the centre $Z([K^\circ, K^\circ])$ are all characteristic in $G$. Therefore, in view of Lemma~\ref{lem:sol}, we can assume that $Z([K^\circ, K^\circ])$ is trivial. In that case, $[K^\circ, K^\circ]$ is a product of (possibly infinitely many) simple compact Lie groups; see Theorem~9.19 in~\cite{Hofmann-Morris} (recalling that the commutator of a compact connected group is ``semisimple'' in the terminology of~\cite{Hofmann-Morris}). Since these simple Lie groups are centre-free, they are topologically simple. In view of Proposition~\ref{prop:simple}, we can therefore assume that $[K^\circ, K^\circ]$ itself is trivial. Now $K^\circ$ is abelian and can be assumed trivial by Lemma~\ref{lem:sol}. We deduce that $K$ is central in $G^\circ$ since any totally disconnected normal subgroup of a connected group is central; appealing one more time to Lemma~\ref{lem:sol}, we can assume $K$ trivial. In other words, $G^\circ$ is a Lie group. Applying Lemma~\ref{lem:sol}, we can assume that its radical as well as its centre are trivial. Therefore $G^\circ$ is a centre-free semi-simple connected Lie group, and thus a product of topologically simple groups. Proposition~\ref{prop:simple} shows that we can assume it to be trivial.

\bigskip
The conclusion of these reductions is that we have brought ourselves to the case where $G$ is totally disconnected, and thus classical structure theory will not help us any further.

\bigskip
Let $G^+$ denote the discrete residual of $G$, i.e. the intersection of all open normal subgroups of $G$. By Theorem~F in~\cite{Caprace-Monod_compact}, $G^+$ is  a characteristic cocompact subgroup of $G$ and  has no non-trivial discrete quotients. In particular, it is contained in every cocompact normal subgroup of $G$ since all compact quotients of $G$ are profinite. By Lemma~\ref{lem:pro}, there exists an element in $G/G^+$ that normally generates it. Let $gG^+$ be the corresponding coset in $G$. Thus, no element of $gG^+$ is contained in any cocompact normal proper subgroup of $G$.

\smallskip
Every non-cocompact closed normal subgroup of $G$ is contained in a maximal one, see Proposition~5.2 in~\cite{Caprace-Monod_compact}. It was proved in~\cite{Caprace-Monod_compact}  that there are only finitely many such maximal non-cocompact closed normal subgroups in $G$ --- indeed, use Corollary~5.1 of~\cite{Caprace-Monod_compact} as in the proof of Theorem~A ibidem. We denote them by $M_1, \ldots, M_k$ and assume that $k\geq 1$ since otherwise $G$ is compact and we are done by Lemma~\ref{lem:pro}. We claim that there are elements in the coset $gG^+$ which do not belong to any of the $M_i$. Such elements are therefore not contained in any proper closed normal subgroup of $G$; that is, any such element normally generates $G$.

\smallskip
 In order to prove the claim, suppose for a contradiction that $gG^+$ is contained in the union of all $M_i$. Then $gG^+\cap M_i$ has non-empty interior in $gG^+$  for some $i$. Thus $G^+\cap g^{-1}M_i$ has some interior point $x$ relatively to $G^+$. Therefore, the identity is an interior point of $G^+\cap x^{-1}g^{-1}M_i$ within $G^+$. Now $x^{-1}g^{-1}M_i = M_i$ and $G^+\cap M_i$ is an open normal subgroup of $G^+$. As $G^+$ has no non-trivial discrete quotients we conclude that $G^+\cap M_i=G^+$ which means $G^+<M_i$. This is absurd because $G^+$ is cocompact in $G$ but $M_i$ is not. This completes the proof of the claim and therefore of Theorem~\ref{thm:main}.\qed

%===================================================================================================
%===================================================================================================
\section{On compact generation}\label{sec:Yves}
Yves de Cornulier has kindly allowed us to present his example of a perfect locally compact group $G$ which is not normally generated by a single element although it has no infinite discrete (group) quotient.

\smallskip\noindent
Assume that $F$ is a finite perfect group with a subgroup $K<F$ such that

\begin{enumerate}[(i)]
\item $F$ is normally generated by $K$,\label{pt:K}
\item $F$ is not normally generated by any single element of $K$.\label{pt:noK}
\end{enumerate}

\noindent
Let $G$ be the group of sequences in $F$ that are ultimately in $K$. The ``restricted product'' locally compact group topology on $G$ is defined by the identity neighbourhood base $U_n< K^\NN$ consisting of sequences that are trivial on the first $n$ coordinates. The perfect group $F^{(\NN)}$ of finitely supported sequences is then dense in $G$, so that $G$ is perfect. Moreover, condition~\eqref{pt:noK} ensures that $G$ cannot be normally generated by a single element.

\smallskip
Consider any discrete quotient of $G$. Its kernel must contain some $U_n$; being normal, it follows by condition~\eqref{pt:K} that this quotient is a quotient of $F^n$, thus it is finite.

\bigskip
It remains to see that a pair $K<F$ as above indeed exists. Let $c$ be an element of order two in a non-abelian finite simple group $S$. Set $F=S^3$ and let $K$ be the (Klein) subgroup generated by $(1,c,c)$ and $(c,c,1)$. Since the maximal normal subgroups of $F$ are the kernels of the three canonical projections $F\to S$, which do not contain $K$, condition~\eqref{pt:K} holds. On the other hand, any element of $K$ is in one of these kernels, whence condition~\eqref{pt:noK}.

%===================================================================================================
%===================================================================================================
%\section{Comments}

%======================================================================================================
%======================================================================================================
\bibliographystyle{../../BIB/abbrv}
\bibliography{../../BIB/ma_bib}

\end{document}